\documentclass[12pt,reqno]{amsart}

\usepackage{amssymb,amsmath,amsthm}
\usepackage{color}

\newtheorem{theorem}{Theorem}[section]
\newtheorem{lemma}[theorem]{Lemma}
\newtheorem{prop}[theorem]{Proposition}

\newtheorem{remark}[theorem]{Remark}

\def \R{\mathbb{R}}

\def \Z{\mathbb{Z}}
\def \Rn{\mathbb{R}^n}

\def \e{\varepsilon}

\def \grad{\nabla}

\def \half{{1}/{2}}

\date{\today}

\numberwithin{equation}{section}

\begin{document}

\title[Radiative Transport UMBLT]{Ultrasound modulated bioluminescence tomography
and controllability of the radiative transport equation}
 
\author[Bal]{Guillaume Bal}
\address{Department of Applied Physics and Applied Mathematics, Columbia University, New York, NY, USA}
\email{gb2030@columbia.edu}

\author[Chung]{Francis J. Chung}
\address{Department of Mathematics, University of Michigan, Ann Arbor, MI, USA}
\email{fjchung@umich.edu}

\author[Schotland]{John C. Schotland}
\address{Department of Mathematics and Department of Physics, University of Michigan, Ann Arbor, MI, USA}
\email{schotland@umich.edu}



\begin{abstract}
We propose a method to reconstruct the density of an optical source in a highly scattering medium from ultrasound-modulated optical measurements. Our approach is based on the solution to a hybrid inverse source problem for the radiative transport equation (RTE). A controllability result for the RTE plays an essential role in the analysis.
\end{abstract}

\maketitle


\section{Introduction}

This paper is concerned with the problem of reconstructing an optical source in a highly-scattering medium. The primary application is to bioluminescence imaging, in which the cells of a model organism are tagged with a light-emitting molecular probe~\cite{contag,vasilis}. The goal is to recover the spatial distribution of the labelled cells from measurements of the emitted light. This is a classical inverse source problem for the radiative transport equation (RTE), which may be formulated as follows.

Let $X$ be a bounded domain in ${\mathbb R}^n$ with smooth boundary $\partial X$, for dimension $n\ge 2$. The specific intensity $u$ obeys the RTE 
\begin{equation}
\label{SourceRTE} 
\theta \cdot \grad u + \sigma(x) u - \int_{S^{n-1}} k(x,\theta,\theta ')u(x,\theta ')\, d \theta' = S(x) \ .
\end{equation}

Here $u(x,\theta)$ is the intensity of light at the point $x\in X$ 
traveling in the direction $\theta\in S^{n-1}$, $S$ is the source (which is taken to be isotropic) and $\sigma$ is the attenuation coefficient, which is taken to be nonnegative. The scattering kernel $k$ is nonnegative and obeys the reciprocity relation $k(x,\theta,\theta')=k(x,-\theta',-\theta)$ and the normalization condition
$\int k(x,\theta,\theta')d\theta' = 1$ for all $\theta\in S^{n-1}$. {We will also assume} that $k$ is invariant under rotations. That is, $k(x,\theta,\theta')=
k(x,\theta\cdot\theta')$, which holds for statistically homogeneous random media.
We define the subsets $\Gamma_{\pm}$ of $\partial X \times S^{n-1}$ by
\begin{equation}
\Gamma_{\pm} = \{(x,\theta) \in \partial X \times S^{n-1} : \pm \theta \cdot n(x)  > 0\} \ ,
\end{equation}
where $n(x)$ is the outward unit normal vector at $x$. The specific intensity obeys the boundary condition $u(x,\theta)=0$ for $(x,\theta)\in \Gamma_-$. 
Thus no light enters the domain except from the source. The inverse source problem is to reconstruct $S$ from boundary measurements of the outgoing specific intensity $u|_{\Gamma_+}$, assuming that $\sigma$ and $k$ are known. This problem has been considered in~\cite{BalTamasan}, where it is shown that it is possible to uniquely recover $S$, provided that $k$ is sufficiently small in a suitable norm. See also~\cite{larsen,siewert} for related work. In~\cite{StefanovUhlmann}, it was shown that the smallness condition can be {replaced by the assumption that $\sigma$ and $k$ belong to $C^2(X)$}. A corresponding stability estimate was also derived. {However, the case of continuous coefficients remains open.

Ultrasound modulated bioluminescence tomography (UMBLT) is a recently proposed imaging modality in which an acoustic wave is used to spatially modulate the optical source, while boundary measurements of the optical field are performed~\cite{BalSchotland_PhysRev2014,huynh_2013}. In this manner, the above inverse source problem is converted into a so-called hybrid inverse problem, where an external field is used to control the material properties of a medium of interest, which is then probed by a second field. See~\cite{ammari_2008,Bal_IP2009,BalSchotland_2010,bal_2012_a,bal_2010,bal_2012_b,bal_2012_c,bal_2013,capdeboscq_2009,gebauer_2009,kuchment,kuchment_2012,monard_2011,mclaughlin_2004,mclaughlin_2010,nachman} for examples of hybrid inverse problems in other physical settings.

The inverse problem of UMBLT was first studied in the special case of the diffusion approximation (DA) to the RTE~\cite{BalSchotland_PhysRev2014}. It was shown that it is possible to uniquely reconstruct the source with Lipschitz stability. However, the DA breaks down in strongly-absorbing or weakly-scattering media, near boundaries and on small length scales. In this paper, we consider the inverse problem of UMBLT in the transport regime. Once again, we find that the source can be recovered with Lipschitz stability by a constructive procedure.

The forward problem of UMBLT is formulated as follows. Suppose that a standing acoustic pressure wave with modulation amplitude of the form $\cos(q \cdot x + \varphi)$ is incident on a highly-scattering medium, where $q$ is the wave vector
and $\varphi$ is the phase of the wave. Following~\cite{BalSchotland_2010,BalSchotland_PhysRev2014}, we find that the the coefficients $\sigma$ and $k$ and the source $S$ are modulated thus forming new coefficients $\sigma_{\e}$, $k_{\e}$ and $S_{\e}$, which are given by
\begin{eqnarray}
\sigma_{\e}(x) &=& (1 + \e \cos(q\cdot x + \varphi)) \sigma(x) \ , \\
k_{\e}(x) &=& (1 + \e \cos(q\cdot x + \varphi)) k(x) \ , \\
S_{\e}(x) &=& (1 + \e \cos(q\cdot x + \varphi)) S(x) \ , 
\end{eqnarray}
where {$0 < \e\ll 1$} is the dimensionless amplitude of the acoustic wave.  The RTE thus becomes
\begin{equation}
\label{ModulatedxourceRTE}
\theta \cdot \grad u_\epsilon + \sigma_{\e}(x) u - \int_{S^{n-1}} k_{\e}(x,\theta,\theta ')u_{\e}(x,\theta ') \, d \theta ' = S_{\e}(x) \ ,
\end{equation}
where the dependence of the specific intensity $u$ on $\epsilon$ has been made explicit. The inverse problem now consists of deducing $S$ from boundary measurements of $u_{\e}$ when $q$ and $\varphi$ are varied. 

Throughout this paper we will make the following assumptions on the domain $X$ and the coefficients $\sigma$ and $k$ that appear in the RTE.  First, we will assume that $X$ is a bounded subset of $\Rn$ with diameter $\tau$.  Next, we will assume that $\sigma$, $k$, and $S$ are continuous. Finally, in order to ensure solvability of the RTE, let  
\begin{equation}
\rho = \left\| \int_{S^{n-1}}  k(x,\theta, \theta') d\theta'  \right\|_{L^{\infty}(X \times S^{n-1})},
\end{equation}
and assume that one of the following inequalities holds:
\begin{equation}\label{AbsorptionCondition}
\sigma - \rho \geq \alpha,
\end{equation}
for some positive constant $\alpha$, or
\begin{equation}\label{SmallnessCondition}
\tau \rho < 1.
\end{equation}    

Under the above conditions, the equation \eqref{ModulatedxourceRTE} with boundary condition $u_{\e}|_{\Gamma_-} = 0$ has a unique continuous solution $u_{\e}$ (see Section 2 for details), and we can define $\Lambda^{\e}_{S}(q, \varphi): \R^n \times\{0,\frac{\pi}{2}\} \rightarrow C(\Gamma_{+})$ to be the map given by
\begin{equation}
\Lambda^{\e}_{S}(q, \varphi) = \frac{1}{\e}(u_{\e}-u_0)|_{\Gamma_{+}} \ ,
\end{equation}
if $\e \neq 0$. (Recall that $u_{\e}$ is defined in part by the choice of $q$ and $\varphi$.)  We will also need to make separate use of the measurements when $\e = 0$, that is without ultrasound modulation.  Therefore we will define
\begin{equation}
\Lambda^{0}_{S} = u_0|_{\Gamma_{+}}.
\end{equation}

We have the following theorem. 

\begin{theorem}
\label{IPTheorem}
Suppose $X$, $\sigma$ and $S$ are as above and that $k$ is invariant under rotations. Then the map $S \mapsto \{\Lambda^{\e}_S: \e \geq 0\}$ is injective.  Moreover, we have the stability estimate
\begin{equation}
C\|S_1 - S_2\|_{L^{\infty}(X)} \leq \|\Lambda^0_{S_1} - \Lambda^0_{S_2}\|_{C(\Gamma_{+})}+ \|\Lambda^{\e}_{S_1} - \Lambda^{\e}_{S_2}\|_{L^1(\R^n \times\{0,\frac{\pi}{2}\},C(\Gamma_{+}))} + O(\e)
\end{equation}
where $C$ depends only on $X$, $k$, and $\sigma$.  
\end{theorem}

The proof of Theorem \ref{IPTheorem} is constructive. That is, we will also give an algorithm by which $S$ can be reconstructed from knowledge of $\Lambda^{\e}_S$.

\begin{remark}
  Note that as compared to the result of~\cite{StefanovUhlmann}, the above stability estimate does not depend on derivatives of {the coefficients $\sigma$ and $k$}.
\end{remark}

A key ingredient in the proof of Theorem \ref{IPTheorem} is the following controllability result for the RTE, which is of interest in its own right.  

\begin{theorem}\label{ControlTheorem}
Suppose $X$, $\sigma$, and $k$ are as given above.  Then for any point $x_0 \in X$ and any continuous function $h$ on $L^p(S^{n-1})$, there is a function $g \in L^p(S^{n-1},L^{\infty}(\partial X))$ such that the boundary value problem
\begin{eqnarray*}
\theta \cdot \grad v + \sigma v &=& \int_{S^{n-1}} k(x,\theta,\theta ')v(x,\theta ')\, d \theta' \\
                v|_{\Gamma_{-}} &=& g|_{\Gamma_{-}} \\
\end{eqnarray*}
has a unique solution $v\in L^p(S^{n-1},L^{\infty}(X))$ which is continuous in a neighbourhood of $x_0$, and satisfies the property that $v(x_0, \theta) = h(\theta)$, for all $\theta \in S^{n-1}$.  Moreover, for any $1 \leq p \leq \infty$,
\begin{equation}
\|v\|_{L^p(S^{n-1},L^{\infty}(X))}  +  \|g\|_{L^p(S^{n-1},L^{\infty}(\partial X))} \leq C\|h\|_{L^p(S^{n-1})}.
\end{equation}
where $C$ depends only on $X$, $\sigma$, and $k$.  
\end{theorem}

The proof of Theorem \ref{ControlTheorem} is also constructive. Note that other control theory results for the time-dependent RTE have had applications to inverse problems as well. See for example~\cite{Acosta1, Acosta2}.

The remainder of this paper is organized as follows. In Section 2, we prove a regularity result for the RTE which will be needed in the proofs of Theorem \ref{IPTheorem} and Theorem \ref{ControlTheorem}. We then prove Theorem \ref{IPTheorem} in Section 3, using Theorem \ref{ControlTheorem}. Finally,  we will prove Theorem \ref{ControlTheorem} in Section 4.  


\section{Regularity}

In order to prove Theorem \ref{IPTheorem} and Theorem \ref{ControlTheorem}, we need the following regularity result for the RTE.  

\begin{theorem}\label{RegularityTheorem}
Let $f_{-} \in C(\Gamma_{-})$.  Under the conditions on $X$, $\sigma$, $k$, and $S$ given in Section 1, the equation \eqref{SourceRTE} has a unique continuous solution $u$ with the boundary condition
\begin{equation}
u|_{\Gamma_{-}} = f_{-}.
\end{equation}
Moreover, if \eqref{AbsorptionCondition} holds, then 
\begin{equation}\label{AbsorptionEst}
\|u\|_{L^p(S^{n-1}, C(X))} \leq \frac{1}{\alpha}\big( (\rho + \alpha) \|f_{-}\|_{L^p(S^{n-1}, C(\partial X))} + \|S\|_{L^p(S^{n-1}, C(X))} \big).
\end{equation}
If instead \eqref{SmallnessCondition} holds, then 
\begin{equation}\label{SmallnessEst}
\|u\|_{L^p(S^{n-1}, C(X))} \leq \frac{1}{1 - \tau \rho} \big( \|f_{-}\|_{L^p(S^{n-1}, C(\partial X))} + \tau\|S\|_{L^p(S^{n-1}, C(X))} \big).
\end{equation}
In both cases, the estimate given is valid for $1 \leq p \leq \infty$.
\end{theorem}

The proof of Theorem \ref{RegularityTheorem} is close to that of similar results in~\cite{ChoulliStefanov}; see also~\cite{DautrayLions}.  Let
\begin{equation}
\tau_{\pm}(x,\theta) = \mathrm{min}\{t \geq 0| x \pm t\theta \in \partial X\}
\end{equation}  
and 
\begin{eqnarray}
A_1 u &=& \sigma u \ , \\
A_2 u &=& -\int_{S^{n-1}} k(x,\theta, \theta')u(x,\theta') \, d\theta '  \ ,\\
A u &=& A_1 u + A_2 u \ , \\
T_0 u &=& \theta \cdot \grad u \ , \\
T_1 u &=& T_0 + A_1. 
\end{eqnarray}
In addition, let 
\begin{equation}
B(t,x,\theta) = \exp\left(-\int_0^{t}\sigma(x - s\theta,\theta)ds\right).
\end{equation}
Note that $B$ is continuous in each variable and $|B(t,x,\theta)| < 1$.  Moreover, if \eqref{AbsorptionCondition} holds, then we have the improved estimate
\begin{equation}
|B(t,x,\theta)| \leq e^{-t(\rho+\alpha)}.
\end{equation}
Finally, if $f_{-} \in C(\Gamma_{-})$, we define $Jf_{-}$ to be the function on $X \times S^{n-1}$ defined by 
\begin{equation}
Jf_{-} = B(\tau_{-}(x,\theta),x,\theta) f_{-}(x - \tau_{-}(x,\theta)\theta, \theta).
\end{equation}
Note that $Jf_{-}$ is continuous, $T_1 Jf_{-} = 0$, and $Jf_{-}|_{\Gamma_{-}} = f_{-}$.  Moreover, for any fixed $\theta\in S^{n-1}$, 
\begin{equation}
\|Jf_{-}\|_{C(X)} \leq \|f_{-}\|_{C(\partial X)},
\end{equation}
and so
\begin{equation}
\|Jf_{-}\|_{L^p(S^{n-1},C(X))} \leq \|f_{-}\|_{L^p(S^{n-1},C(\partial X))}.
\end{equation}

\begin{proof}[Proof of Theorem \ref{RegularityTheorem}]

Suppose $u$ satisfies \eqref{SourceRTE} with the boundary condition $u|_{\Gamma_{-}} = f_{-}$.  Then on $X \times S^{n-1}$, we have in the above notation that 
\begin{equation}
(T_1 + A_2)u = S.
\end{equation}
Thus for any $t \in [0, \tau_{-}(x,\theta)]$, we have 
\begin{equation}\label{PreIntegralForm}
B(t,x,\theta)[(T_1 + A_2)u](x - t\theta, \theta) = B(t,x,\theta)S(x - t\theta). 
\end{equation}
Now let
\begin{equation}
T_1^{-1}f(x,\theta) = \int_0^{\tau_{-}(x,\theta)}B(t,x,\theta)f(x - t\theta, \theta)dt.
\end{equation}
Then $T_1^{-1}f$ is continuous if $f$ is. Moreover, for any fixed $\theta$,  
\begin{equation}
\|T_1^{-1}f\|_{C(X)} \leq \tau \|f\|_{C(X)},
\end{equation}
and if we assume \eqref{AbsorptionCondition}, then
\begin{eqnarray}
\|T_1^{-1}f\|_{C(X)}  &\leq& \|f\|_{C(X)} \int_0^{\tau_{-}(x,\theta)}|B(t,x,\theta)|dt  \\
                      &\leq& \|f\|_{C(X)} \int_0^{\tau_{-}(x,\theta)}e^{-t(\rho + \alpha)}dt  \\
							        &\leq& (\rho + \alpha)^{-1} \|f\|_{C(X)}. 
\end{eqnarray}
Therefore it follows that 
\begin{equation}
\|T_1^{-1}f\|_{L^p(S^{n-1},C(X))} \leq \tau \|f\|_{L^p(S^{n-1},C(X))},
\end{equation}
and if  \eqref{AbsorptionCondition} holds, then
\begin{equation}
\|T_1^{-1}f\|_{L^p(S^{n-1},C(X))} \leq  (\rho + \alpha)^{-1} \|f\|_{L^p(S^{n-1},C(X))}.
\end{equation}

If we now integrate \eqref{PreIntegralForm} in $t$ from $0$ to $\tau_{-}(x,\theta)$, then we obtain after some calculation the result 
\begin{equation}\label{IntegralRTE}
(I+T_1^{-1}A_2)u = Jf_{-} - T_1^{-1}S.
\end{equation}
Therefore, if $u$ solves \eqref{SourceRTE} with boundary condition $u|_{\Gamma_{-}} = f_{-}$, then $u$ solves \eqref{IntegralRTE}.  Moreover, if we apply $T_1$ to both sides of \eqref{IntegralRTE}, then we get \eqref{SourceRTE}, and since $\tau_{-}(x,\theta) = 0$ on $\Gamma_{-}$,  it follows that on $\Gamma_{-}$, \eqref{IntegralRTE} reads $u = f_{-}$.  Therefore $u$ solves \eqref{SourceRTE} with the boundary condition $u|_{\Gamma_{-}} = f_{-}$ if and only if $u$ solves \eqref{IntegralRTE}.  

Now if $K = T_1^{-1}A_2$, then \eqref{IntegralRTE} reads
\begin{equation}
(I+K)u = Jf_{-} - T_1^{-1}S.
\end{equation}
Moreover, if \eqref{AbsorptionCondition} holds, then
\begin{equation}
\|Ku\|_{L^p(S^{n-1},C(X))}  \leq \frac{\rho}{\rho + \alpha} \|u\|_{L^p(S^{n-1},C(X))}  < \|u\|_{L^p(S^{n-1},C(X))} ,
\end{equation}
and if \eqref{SmallnessCondition} holds, then
\begin{equation}
\|Ku\|_{L^p(S^{n-1},C(X))}  \leq \tau \rho\|u\|_{L^p(S^{n-1},C(X))}  < \|u\|_{L^p(S^{n-1},C(X))} .
\end{equation}
In either case, $(I+K)u$ has a unique inverse which can be written using Neumann series.  The series
\begin{equation}
u = (I - K + K^2 - K^3 + \cdots)( Jf_{-} - T_1^{-1}S)
\end{equation}
converges uniformly in $(x,\theta)$, since we can set $p = \infty$, and hence yields a continuous function $u$, which must satisfy the inequality \eqref{AbsorptionEst} in the case that \eqref{AbsorptionCondition} holds, and \eqref{SmallnessEst} in the case that \eqref{SmallnessCondition} holds. Thus $u$ solves \eqref{IntegralRTE} and thus solves \eqref{SourceRTE}.  

\end{proof}

As a remark, note that if boundary source $f_{-}\in L^{p}(S^{n-1}, L^{\infty}(\partial X))$  but is not continuous, then the same argument as above will give a solution $u\in L^{p}(S^{n-1}, L^{\infty}(X))$.   


\section{Inverse Problem}

We recall that the inverse problem consists of deducing the source $S$ from boundary measurements of the specific intensity $u_{\e}$, assuming the coefficients $\sigma$ and $k$ are known. Here we also assume that $k$ is invariant under rotations. To proceed, we suppose that $v$ is a continuous solution to the adjoint equation
\begin{equation}\label{AdjointRTE}
-\theta \cdot \grad v + \sigma v = \int_{S^{n-1}} k(x,\theta\cdot\theta')v(x,\theta ')\, d \theta' \ ,
\end{equation}
obeying the boundary condition $v|_{\Gamma_+}$ prescribed. Since $k$ and $\sigma$ are known, we will suppose that $v$ is known.  We now make use of the integration by parts identity
\begin{equation}\label{IbyP}
\int_{X} v \theta \cdot \grad u_{\e} \, dx = - \int_{X} u_{\e} \theta \cdot \grad v\, dx + \int_{\partial X} u_{\e} v \, n \cdot \theta \, dx 
\end{equation}
where $n$ is the outward normal vector to $\partial X$.  If $v$ satisfies \eqref{AdjointRTE} and $u_{\e}$ satisfies \eqref{ModulatedxourceRTE}, then \eqref{IbyP} becomes
\begin{eqnarray}
& & -\int_{X} v \sigma_{\e} u_{\e}\, dx +\int_{X} v(x, \theta)\int_{S^{n-1}}k_{\e}(x,\theta\cdot\theta ') u_{\e}(x,\theta ')\, d\theta ' \, dx + \int_{X} v S_{\e} \, dx \\
&=& -\int_{X} u_{\e} \sigma v\, dx +\int_{X} u_{\e}(x, \theta)\int_{S^{n-1}} k(x,\theta\cdot\theta ')v(x,\theta ') \, d\theta '\,  dx \\
& & +\int_{\partial X} u_{\e} v \, n \cdot \theta \, dx.
\end{eqnarray} 
Rearranging and integrating over $\theta$ gives 
\begin{eqnarray}
& & \int_{S^{n-1}} \int_{\partial X} u_{\e} v \, n \cdot \theta \, dx \, d\theta \\
&=& \int_{S^{n-1}} \int_{X} (\sigma - \sigma_{\e})u_{\e} v\,  dx \, d\theta + \int_{S^{n-1}} \int_{X} v S_{\e}\, dx \, d\theta \\
& & +\int_{X}\int_{S^{n-1}}\int_{S^{n-1}}v(x, \theta)k_{\e}(x,\theta\cdot\theta ') u_{\e}(x,\theta ')\, d\theta '\, d\theta \, dx \\
& & -\int_{X}\int_{S^{n-1}}\int_{S^{n-1}}u_{\e}(x, \theta)k(x,\theta\cdot\theta ') v(x,\theta ') \, d\theta '\, d\theta \, dx. 
\end{eqnarray} 
We can now write the above as
\begin{equation}\label{PreAsymptotic}
\begin{split}
\int_{S^{n-1}} \int_{\partial X} u_{\e} v \, n \cdot \theta \, dx \, d\theta = & \int_{S^{n-1}} \int_{X} (\sigma - \sigma_{\e})u_{\e} v\,  dx \, d\theta + \int_{S^{n-1}} \int_{X} v S_{\e}\, dx \, d\theta \\
 & \, \, \, +\int_{X}\int_{S^{n-1}}\int_{S^{n-1}}(k_{\e}-k)(x,\theta\cdot\theta ')v(x, \theta)u_{\e}(x,\theta ')\, d\theta '\, d\theta \, dx. \\
\end{split}
\end{equation}
Since we can measure $u_{\e}|_{\partial X}$, and $v$ is assumed to be known, the left hand side of the above equation is a known quantity.  Now if $\e = 0$, then writing $u = u_0$, we have 
\begin{equation}
\int_{S^{n-1}} \int_{\partial X} u v \, n \cdot \theta \, dx \, d\theta = \int_{S^{n-1}} \int_{X} v S\, dx \, d\theta.
\end{equation}
This fails to take advantage of the ultrasound modulation, though.  If we define $u^1_{\e}$ by 
\begin{equation}
\e u^1_{\e} = u_{\e} - u,
\end{equation}
then $u^1_{\e}$ satisfies the equation
\begin{eqnarray}
\e \theta \cdot \grad u^1_{\e} &+& \e \sigma_{\e} u^1_{\e} = \e\int_{S^{n-1}}k_{\e}(x,\theta,\theta ')u^1_{\e}(x,\theta ') \, d\theta ' \\
                             & &+S_{\e}-S+\int_{S^{n-1}}(k_{\e}-k)(x,\theta\cdot\theta ')u(x,\theta ') \, d\theta '+(\sigma -\sigma_{\e})u. 
\end{eqnarray}
Since the whole second line above is $O(\e)$, it follows from Theorem \ref{RegularityTheorem} that $\|u^1_{\e}\|$ is $O(1)$.  Therefore 
\begin{eqnarray}
& & \int_{S^{n-1}} \int_{\partial X} u^1_{\e} v \, n \cdot \theta \, dx \, d\theta \\
&=& -\int_{S^{n-1}} \int_{X} \cos(q\cdot x + \varphi)\sigma u v\,  dx \, d\theta + \int_{S^{n-1}} \int_{X} v \cos(q\cdot x + \varphi)S\, dx \, d\theta \\
& & +\int_{X}\int_{S^{n-1}}\int_{S^{n-1}}\cos(q\cdot x + \varphi)k(x,\theta\cdot\theta ')v(x, \theta)u(x,\theta ')\, d\theta '\, d\theta \, dx + O(\e) . \hskip 20pt
\end{eqnarray}
Here the above left hand side is still known.  If we now vary $q$ and $\varphi$ and take $\e \rightarrow 0$, we can obtain the Fourier transform of the quantity $H_v$ defined by
\begin{eqnarray}
H_v(x) &=& -\int_{S^{n-1}} \sigma u v\, d\theta + \int_{S^{n-1}} v S\, d\theta \\
     & & +\int_{S^{n-1}}\int_{S^{n-1}}k(x,\theta\cdot\theta ')v(x, \theta)u(x,\theta ')\, d\theta '\, d\theta.
\end{eqnarray}
We note that the above expression for $H_v$ has two unknowns, $u$ and $S$.
However, equation \eqref{SourceRTE} also relates $u$ and $S$.  If we use \eqref{SourceRTE} to substitute for $S$ in the expression for $H_v$, nearly everything cancels and we find that
\begin{equation}
\label{Fxnal}
H_v(x) = \int_{S^{n-1}} v(x,\theta) \theta \cdot \grad u(x,\theta) \, d\theta.
\end{equation}

\begin{remark} 
We note that $H_v(x)$ is a so-called internal functional; it is known for every point $x\in X$. Such internal functionals, which are determined from boundary measurements, play the role of internal measurements and are a generic feature of hybrid inverse problems.
\end{remark}

The inverse problem now consists of recovering $S$ from $H_v$. Suppose $x_0 \in X$. Then by Theorem \ref{ControlTheorem}, we can arrange for $v(x_0,\theta)$ to be any continuous function in $\theta$.  Therefore knowing $H_v(x_0)$ for all $v$ solving \eqref{AdjointRTE} is equivalent to knowing $\theta \cdot \grad u(x_0,\theta)$ for each $\theta \in S^{n-1}$.  Since this can be done for any $x_0 \in X$, we can recover from $H_v$ knowledge of $\theta \cdot \grad u(x,\theta)$ for all  $(x,\theta) \in X \times S^{n-1}$.  Then using the formula
\begin{equation}
u(x, \theta) = u(x + \tau_{+}\theta,\theta) - \int_0^{\tau_{+}}\theta \cdot \grad u(x + t\theta,\theta) dt
\end{equation}
and the fact that we know $u|_{\Gamma_{+}}$, we see that we can recover $u(x,\theta)$ at each point in $X \times S^{n-1}$.  Then using \eqref{SourceRTE}, we can recover $S$.  

We now examine the stability of the reconstruction described above and thereby prove Theorem~\ref{IPTheorem}.
\begin{proof}[Proof of Theorem~\ref{IPTheorem}.]
Suppose $S_1$ and $S_2$ are different sources.  Let $u_{1,\e}$ and $u_{2,\e}$ be the solutions of \eqref{ModulatedxourceRTE} corresponding to $S_1$ and $S_2$.  Then since equation \eqref{PreAsymptotic} applies to $u_{1,\e}$ and $u_{2,\e}$, we have
\begin{equation}\label{PreAsymptoticStability}
\begin{split}
  &\int_{S^{n-1}} \int_{\partial X} (u_{1,\e}-u_{2,\e}) v \, n \cdot \theta \, dx \, d\theta \\
= & \, \int_{S^{n-1}} \int_{X} (\sigma - \sigma_{\e})(u_{1,\e}-u_{2,\e}) v\,  dx \, d\theta + \int_{S^{n-1}} \int_{X} v (S_{1,\e} - S_{2,\e})\, dx \, d\theta \\
 & \, \, \, +\int_{X}\int_{S^{n-1}}\int_{S^{n-1}}(k_{\e}-k)(x,\theta\cdot\theta ')v(x, \theta)(u_{1,\e}-u_{2,\e})(x,\theta ')\, d\theta '\, d\theta \, dx. \\
\end{split}
\end{equation}
If $\e = 0$ we have
\begin{equation}
\int_{S^{n-1}} \int_{\partial X} (u_{1,0}-u_{2,0}) v n \cdot \theta \, dx \, d\theta = \int_{S^{n-1}} \int_{X} v (S_{1} - S_{2})\, dx \, d\theta.
\end{equation}
Then \eqref{PreAsymptoticStability} becomes
\begin{equation}
\begin{split}
  &\int_{S^{n-1}} \int_{\partial X} ((u_{1,\e} - u_{1,0}) -(u_{2,\e} - u_{2,0})) v \, n \cdot \theta \, dx \, d\theta \\
= & \, \int_{S^{n-1}} \int_{X} (\sigma - \sigma_{\e})(u_{1,\e}-u_{2,\e}) v\,  dx \, d\theta   + \int_{S^{n-1}} \int_{X} v ((S_{1,\e}- S_1) - (S_{2,\e}-S_2))\, dx \, d\theta \\\\
 & \, \, \, +\int_{X}\int_{S^{n-1}}\int_{S^{n-1}}(k_{\e}-k)(x,\theta\cdot\theta ')v(x, \theta)(u_{1,\e}-u_{2,\e})(x,\theta ')\, d\theta '\, d\theta \, dx. 
\end{split}
\end{equation}
Now note that 
\begin{eqnarray}
\int_{S^{n-1}} \int_{\partial X} (u_{1,\e} - u_{1,0}) v \, n \cdot \theta \, dx \, d\theta &=& \int_{\Gamma_{+}}(u_{1,\e} - u_{1,0}) v \, n \cdot \theta \, dx \, d\theta \\
 &=& \e \int_{\Gamma_{+}}\Lambda^{\e}_{S_1} v \, n \cdot \theta \, dx \, d\theta. \end{eqnarray}
Therefore we can rewrite \eqref{PreAsymptoticStability} as 
\begin{equation}
\begin{split}
  &\int_{S^{n-1}} \int_{\partial X} (\Lambda^{\e}_{S_1}-\Lambda^{\e}_{S_2}) v\, n \cdot \theta \, dx \, d\theta \\
= & \, \int_{S^{n-1}} \int_{X} \cos(q\cdot x + \varphi)\sigma(u_{1}-u_{2}) v\,  dx \, d\theta + \int_{S^{n-1}} \int_{X} v \cos(q\cdot x + \varphi)(S_1 - S_2)\, dx \, d\theta\\
 & \, \, \, +\int_{X}\int_{S^{n-1}}\int_{S^{n-1}}\cos(q\cdot x + \varphi)k(x,\theta\cdot\theta ')v(x, \theta)(u_1 - u_2)(x,\theta ')\, d\theta '\, d\theta \, dx +O(\e).
\end{split}
\end{equation}
Using the reasoning described above, we find that
\begin{equation*}
\int_{S^{n-1}} \int_{\partial X} (\Lambda^{\e}_{S_1}-\Lambda^{\e}_{S_2}) v \, n \cdot \theta \, dx \, d\theta \\
= \hat{H}_{v,1} - \hat{H}_{v,2} +O(\e).
\end{equation*}
It follows that for a fixed $v$ solving \eqref{AdjointRTE}, we have
\begin{equation}
\|\Lambda^{\e}_{S_1}-\Lambda^{\e}_{S_2}\|_{L^1(\Rn \times \{0,\frac{\pi}{2}\}, C(\partial X))}\|v\|_{L^1(S^{n-1},C(\partial X))} + O(\e) \geq \|\hat{H}_{v,1} - \hat{H}_{v,2}\|_{L^1(\Rn)},
\end{equation}
which means that 
\begin{equation}
\|H_{v,1} - H_{v,2}\|_{L^{\infty}(X)} \leq \|\Lambda^{\e}_{S_1}-\Lambda^{\e}_{S_2}\|_{L^1(\Rn \times \{0,\frac{\pi}{2}\}, C(\partial X))}\|v\|_{L^1(S^{n-1},C(\partial X))} + O(\e).
\end{equation}
Now by Theorem \ref{ControlTheorem}, we can choose $v$ such that $v(x_0,\theta)$, as a function of $\theta$, is an approximation of identity centered at $\theta_0$.  Moreover, we have the estimate
\begin{equation}
\|v\|_{L^1(S^{n-1},C(X))} \leq c\|v(x_0,\theta)\|_{L^1(S^{n-1})} \leq c,
\end{equation}
where $c$ only depends on $X, \sigma$, and $k$. Then we obtain the estimate
\begin{equation}
|\theta_0 \cdot \grad(u_1 - u_2)(x_0,\theta_0)| \leq  c\|\Lambda^{\e}_{S_1}-\Lambda^{\e}_{S_2}\|_{L^1(\Rn \times \{0,\frac{\pi}{2}\}, C(\partial X))}+ O(\e).
\end{equation}
It follows by integration that  
\begin{equation}
\tau^{-1}\|u_1 - u_2\|_{L^{\infty}(X \times S^{n-1})} \leq c\left( \|\Lambda^0_{S_1} - \Lambda^0_{S_2}\|_{C(\Gamma_{+})}+\|\Lambda^{\e}_{S_1}-\Lambda^{\e}_{S_2}\|_{L^1(\Rn \times \{0,\frac{\pi}{2}\}, C(\partial X))} \right)+ O(\e)
\end{equation}
and thus
\begin{equation}
C\|S_1 - S_2\|_{L^{\infty}(X)} \leq \|\Lambda^0_{S_1} - \Lambda^0_{S_2}\|_{C(\Gamma_{+})}+ \|\Lambda^{\e}_{S_1}-\Lambda^{\e}_{S_2}\|_{L^1(\Rn \times \{0,\frac{\pi}{2}\}, C(\partial X))}+ O(\e) ,
\end{equation}
as claimed.
\end{proof}

Note that if we are interested in the stability of reconstructing the source from the functional $H_v$, we obtain the estimate
\begin{equation}
C\|S_1 - S_2\|_{L^{\infty}(X)} \leq \|H_{v,1} - H_{v,2}\|_{L^{\infty}(X)} + \|\Lambda^0_{S_1} - \Lambda^0_{S_2}\|_{C(\Gamma_{+})}.
\end{equation}
In particular, the $O(\e)$ error disappears, since this comes from the problem of recovering $H_v$ from $\Lambda^{\e}_S$.  This is worth recording as a proposition, for comparison to the equivalent stability result in ~\cite{BalSchotland_PhysRev2014}.  

\begin{prop}
If $S_1, S_2$ are continuous functions on $X \times S^{n-1}$, then for proper choice of $v$, 
\begin{equation}
C\|S_1 - S_2\|_{L^{\infty}(X)} \leq \|H_{v,1} - H_{v,2}\|_{L^{\infty}(X)} + \|\Lambda^0_{S_1} - \Lambda^0_{S_2}\|_{C(\Gamma_{+})}.
\end{equation}
where $H_{v,1}, H_{v,2}$ are the functionals defined by $S_1$ and $S_2$ in terms of \eqref{Fxnal}.
\end{prop} 


\section{Controllability}

The goal of this section is to prove Theorem \ref{ControlTheorem}.  First we need the following lemma, which allows us to propagate solutions to larger and larger sets.

\begin{lemma}\label{PropagationLemma}
Suppose $X_1 \Subset  X_2 \subset X$ are concentric open balls and $X_{12} = X_2 \setminus \bar{X_1}$.  Let 
\begin{equation}
\Gamma_i = \partial X_i \times S^{n-1} \ ,
\end{equation}  
for $i = 1,2$. Now suppose that $u_1$ is continuous on $\bar{X_1}$, and solves 
\begin{equation}\label{PropagationRTE}
\theta \cdot \grad u_1 + \sigma(x) u_1 = \int_{S^{n-1}} k(x,\theta,\theta ')u(x,\theta ')\, d \theta' 
\end{equation}
on $X_1$. Then for any $p$ with $1 \leq p \leq \infty$, there exists a solution $u \in L^p(S^{n-1},L^{\infty}(X_2))$ to \eqref{PropagationRTE} on $X_2$ such that $u = u_1$ on $X_1$. Moreover 
\begin{equation}\label{PropagationEst}
\|u\|_{L^p(S^{n-1},L^{\infty}(X_2))} \leq C\|u_1\|_{L^p(S^{n-1}, C(X_1))},
\end{equation}
where $C$ is a constant that depends only on $\sigma$, $k$, and the diameters of $X_1$ and $X_2$.  
\end{lemma}

The proof of Lemma \ref{PropagationLemma} requires a trace theorem.  Similar trace theorems can be found in ~\cite{Cessenat1985}, ~\cite{Cessenat1984}, as well as ~\cite{DautrayLions}, ~\cite{ChoulliStefanov}, and ~\cite{Bal_IP2009}.  To state the trace theorem, we will once again use the notation introduced in the beginning of Section 2.

\begin{lemma}\label{TraceLemma}
If $u, T_0 u \in L^p(S^{n-1},L^{\infty}(X))$, then for $1 \leq p \leq \infty$
\begin{equation}
\|u\|_{L^p(S^{n-1},L^{\infty}(\partial X))} \leq \tau\|T_0 u\|_{L^p(S^{n-1},L^{\infty}(X))} + \|u\|_{L^p(S^{n-1},L^{\infty}(X))}.
\end{equation} 
\end{lemma}

\begin{proof}
For almost every $(x,\theta)$, we can write  
\begin{equation}
u(x \pm \tau_{\pm}(x,\theta)\theta, \theta) = \int_0^{\tau_{\pm}(x,\theta)} T_0 u(x \pm t\theta,\theta)dt + u(x,\theta) .
\end{equation}
Now $(x \pm \tau_{\pm}(x,\theta)\theta, \theta ) \in \Gamma_{\pm}$ and $\partial X \times S^{n-1}  = \Gamma_{+} \cup \Gamma_{-}$ up to a set of measure zero. Thus for almost every $\theta$,   
\begin{equation}
\|u(\cdot, \theta)\|_{L^{\infty}(\partial X)} \leq \tau \|T_0u(\cdot,\theta)\|_{L^{\infty}(X)} + \|u(\cdot, \theta)\|_{L^{\infty}(X)}.
\end{equation}
The lemma follows by taking $L^p$ norms in $\theta$.  
\end{proof}

Now we can prove Lemma \ref{PropagationLemma}.  
\begin{proof}[Proof of Lemma \ref{PropagationLemma}]
First, we will show that we can construct a function $u_{12} \in L^p(S^{n-1},L^{\infty}(X_{12}))$ which solves \eqref{PropagationRTE} with the boundary condition $u_{12}|_{\Gamma_1} = u_1|_{\Gamma_1}$. To do this, we will begin by picking a continuous function $u_{12}^0$ on $X_{12} \times S^{n-1}$ which satisfies the boundary condition $u_{12}^0|_{\Gamma_1} = u_1|_{\Gamma_1}$ and the estimate
\begin{equation}
\|u_{12}^0\|_{L^p(S^{n-1}, C(X_{12}))} \leq \|u_{1}\|_{L^p(S^{n-1}, C(\partial X_1))}. 
\end{equation}  
It follows from Lemma \ref{TraceLemma} that
\begin{equation}
\|u_{12}^0\|_{L^p(S^{n-1}, C(X_{12}))} \leq c \|u_{1}\|_{L^p(S^{n-1}, C(X_1))} ,
\end{equation}
for some constant $c$ depending on $\sigma$, $k$, and the possibly the diameter of $X_1$.

Now we can define $u^k_{12}$ iteratively by solving the problem
\begin{eqnarray}
T_0u^{k+1}_{12} &=& Au^k_{12} \mbox{ on } X_{12} ,  \\
u^{k+1}_{12}|_{\Gamma_1} &=& u_1|_{\Gamma_1},
\end{eqnarray}
The solution to the general problem 
\begin{eqnarray}
T_0 w &=& Q \mbox{ on } X_{12} , \\
w|_{\Gamma_1} &=& h|_{\Gamma_1}. 
\end{eqnarray}
of this form can be constructed explicitly by
\begin{eqnarray}
\nonumber
w(x,\theta) &=& \int_0^{\tau_{-}(x,\theta)} Q(x - t\theta, \theta) dt + h(x - \tau_{-}(x, \theta)\theta, \theta), \mbox{ if } x - \tau_{-}(x,\theta)\theta \in \partial X_1 \\
\nonumber
w(x,\theta) &=& \int_0^{\tau_{+}(x,\theta)} Q(x + t\theta, \theta) dt + h(x + \tau_{+}(x, \theta)\theta, \theta), \mbox{ if } x + \tau_{+}(x,\theta)\theta \in \partial X_1 \\
w(x,\theta) &=& \int_0^{\tau_{-}(x,\theta)} Q(x - t\theta, \theta) dt, \mbox{ otherwise. } 
\end{eqnarray}
Here $\tau_{\pm}$ are defined for $(x,\theta) \in X_{12} \times S^{n-1}$ with respect to the boundary of $X_{12}$.  

Let $\Delta = \max\{\tau_{\pm}(x,\theta) |(x, \theta) \in X_{12}\}$.  When $h \equiv 0$, we obtain from the expressions above that for any fixed $\theta$,
\begin{equation}
\|w(\cdot, \theta)\|_{L^{\infty}(X_{12})} \leq \Delta \|Q(\cdot, \theta)\|_{L^{\infty}(X_{12})}.
\end{equation}
Therefore
\begin{equation}
\|w\|_{L^p(S^{n-1}, L^{\infty}(X_{12}))} \leq \Delta \|Q\|_{L^p(S^{n-1}, L^{\infty}(X_{12}))}.
\end{equation}
Now if $u^k_{12}$ are defined iteratively as described above and $v^k$ are defined by $v^{k+1} = u^{k+1} - u^k$, it then follows that 
\begin{eqnarray}
T_0v^{k+1}_{12} &=& Av^k_{12} \mbox{ on } X_{12}  , \\
v^{k+1}_{12}|_{\Gamma_1} &=& 0.  
\end{eqnarray}
Therefore 
\begin{equation}
\|v^{k+1}\|_{L^p(S^{n-1}, L^{\infty}(X_{12}))} \leq a \Delta \|v^k\|_{L^p(S^{n-1}, L^{\infty}(X_{12}))} ,
\end{equation}
where 
\begin{equation}
a = \rho + \|\sigma\|.
\end{equation}
If $\Delta$ is sufficiently small, then $a\Delta < 1$. Thus the $v^k$ converge geometrically to zero, and so $u_{12}^k$ converge in the $L^p(S^{n-1},L^{\infty}(X_{12}))$ norm to a function $u_{12}$, where  
\begin{equation}
\|u_{12}\|_{L^p(S^{n-1}, L^{\infty}(X_{12}))} \leq \frac{1}{1-a\Delta}\|u^0_{12}\|_{L^p(S^{n-1}, L^{\infty}(X_{12}))} \leq \frac{c}{1-a\Delta}\|u_{1}\|_{L^p(S^{n-1}, L^{\infty}(X_1))}.
\end{equation}
Since $T_0u_{12}^{k+1} = Au_{12}^{k}$, one can check that $T_0u_{12}^k$ converges in $L^{\infty}$ as well, with a similar estimate applying.  It follows that $u_{12}^{k} \rightarrow u_{12}$ in $W  = \{u \in L^p(S^{n-1}, L^{\infty}(X_{12})) : T_0u \in L^p(S^{n-1}, L^{\infty}(X_{12}))\}$.  Then $u_{12}$ has trace $u_{12}|_{\Gamma_1} = u_1|_{\Gamma_1}$ by Lemma \ref{TraceLemma}, so we can let 
\begin{equation}
u(x,\theta) = \left\{ \begin{array}{ll} u_{1}(x,\theta)  & \mbox{ if } x \in X_1  \\
                                        u_{12}(x,\theta) & \mbox{ if } x \in X_{12} \\ \end{array}\right.
\end{equation}
and check that $u$ is a weak solution to \eqref{PropagationRTE}.  It follows that $u$ must be the unique $L^p(S^{n-1}, L^{\infty}(X_2))$ solution to \eqref{PropagationRTE} on $X_2$ with $L^p(S^{n-1}, L^{\infty}(\Gamma_{2,-}))$ boundary conditions guaranteed by Theorem \ref{RegularityTheorem}.  Moreover, it is easy to see that $u$ must satisfy the estimate \eqref{PropagationEst} since $u_{12}$ does.

If $\Delta$ is not small enough to guarantee that $a\Delta < 1$, we can repeat the above process.  Note that if $X_1$ and $X_2$ are concentric balls of diameter $r$ and $r + \delta$, respectively, then 
\begin{equation}
\Delta = 2 \sqrt{2r\delta + \delta^2}.
\end{equation}
Therefore if ${1}/{a}$ is sufficiently small, 
\begin{equation}
\delta = \min \left\{ \frac{1}{10ar}, \frac{1}{10a} \right\}
\end{equation}
would give $a\Delta < 1$.  Since the harmonic series diverges, it follows that only finitely many repetitions of the above process are required to obtain a solution on a ball of any radius, satisfying the correct estimate. Note, however, that this implies that the constant $C$ in the statement of Lemma \ref{PropagationLemma} blows up as $a$ or $\Delta$ go to infinity.

\end{proof}

\begin{proof}[Proof of Theorem \ref{ControlTheorem}]

We begin by considering the case where $\tau a < 1/2$.  Note that this occurs as long as $X$ is sufficiently small. Now suppose $x_0 \in X$, and $h(\theta)$ is continuous.  We want to construct a solution $v$ to the RTE with the property that $v(x_0,\theta) = h(\theta)$. To begin, consider the equation
\begin{eqnarray}
\theta \cdot \grad v_0 + \sigma(x) v_0 &=& \int_{S^{n-1}} k(x,\theta,\theta ')v_0(x,\theta ')\, d \theta' , \\
      v_0(x,\theta)|_{\Gamma_{-}} &=& h(\theta) 
\end{eqnarray}
By Theorem \ref{RegularityTheorem}, the above has a unique continuous solution $v_0$ with the estimate
\begin{equation}
\|v_0\|_{L^p(S^{n-1}, C(X))} \leq \frac{1}{1 - \tau a}\|h\|_{L^p(S^{n-1})}.
\end{equation}
Now let $g_1(\theta) = h(\theta) - v_0(x_0, \theta)$. We can construct iteratively for $j \geq 1$ 
\begin{eqnarray}
\theta \cdot \grad w_j + \sigma(x) w_j &=& \int_{S^{n-1}} k(x,\theta,\theta ')w_j(x,\theta ')\, d \theta' , \\
      w_j(x,\theta)|_{\Gamma_{-}} &=& g_j(\theta), 
\end{eqnarray}
where
\begin{equation}
g_{j+1}(\theta) = g_j(\theta) - w_j(x_0,\theta).
\end{equation}
Now 
\begin{eqnarray}
g_j(\theta) - w_j(x_0, \theta) &=& \int_0^{\tau_{-}(x_0,\theta)} (\theta \cdot \grad w_j)(x_0 - t\theta, \theta)\, dt \\
                                  &=& \int_0^{\tau_{-}(x_0,\theta)} (A w_j)(x_0 - t\theta, \theta) \,dt.
\end{eqnarray}
Then for a fixed $\theta$, 
\begin{equation}
|g_{j+1}(\theta)| \leq \tau \|A w_j(\cdot, \theta)\|_{C(X)}.
\end{equation}
Therefore
\begin{equation}
\|g_{j+1}\|_{L^p(S^{n-1})} \leq \tau \|A w_j(\cdot, \theta)\|_{L^p(S^{n-1},C(X))} \leq \tau a \|w_j\|_{L^p(S^{n-1},C(X))}.
\end{equation}
By Theorem \ref{RegularityTheorem}, it follows that 
\begin{equation}
\|g_{j+1}\|_{L^p(S^{n-1})} \leq \frac{\tau a}{1-\tau a} \|g_j\|_{L^p(S^{n-1})}.
\end{equation}
Since we are assuming that  $\tau a < \half$, we find that $\|g_{j}\|_{L^p(S^{n-1})}$ converges geometrically to $0$. {Using Theorem \ref{RegularityTheorem} again, we get}$ \|w_j\|_{L^p(S^{n-1},C(X))} \leq 2\|g_j\|_{L^p(S^{n-1})}$, so follows that $\|w_j\|_{L^p(S^{n-1},C(X))}$ converges geometrically to $0$, and thus the sum
\begin{equation}
v(x,\theta) = v_0 (x,\theta) + \sum w_j(x,\theta) 
\end{equation}
converges uniformly, by considering $p = \infty$.  The above yields a continuous function $v$, which has the property that $v(x_0, \theta) = h(\theta)$, and which solves 
\begin{equation}
\theta \cdot \grad v + \sigma(x) v = \int_{S^{n-1}} k(x,\theta,\theta ')v(x,\theta ')\, d \theta'
\end{equation}
with the boundary condition 
\begin{equation}
v|_{\Gamma_{-}} = h(\theta) + \sum g_j(\theta)
\end{equation}
and the estimate 
\begin{equation}
\|v\|_{L^p(S^{n-1},C(X))} \leq \frac{1}{1-\tau \alpha}\|h\|_{L^p(S^{n-1})} + \frac{1 - \tau a}{1 - 2\tau a}\|h\|_{L^p(S^{n-1})}.
\end{equation}

Now if $\tau a \geq \half$, then for $x_0 \in X$, we can begin by taking a small ball $X_0 \subset X$ around $x_0$.  If the diameter of $X_0$ is sufficiently small, then $\tau_{X_0} a < \half$, so by the above reasoning, there is a continuous function $v$ on $X_0$ such that $v(x_0, \theta) = h(\theta)$.  Then by Lemma \ref{PropagationLemma}, there is an $L^p(S^{n-1},L^{\infty}(X))$ solution to \eqref{PropagationRTE} on some large ball containing $X$, which is equal to $v$ on $X_0$.  Restricting this to $X$ finishes the proof of Theorem \ref{ControlTheorem}.
\end{proof}   

\begin{remark}
Note that since the constant $C$ in the statement of Lemma \ref{PropagationLemma} blows up as $a$ or $\Delta$ goes to infinity, it follows that the constant $C$ in the statement of Theorem \ref{ControlTheorem} also blows up as $a$ or the diameter of $X$ become large.  Thus the controllability of the RTE becomes worse as we approach the diffusion limit, which is to be expected.  
\end{remark}

\section{Two-Dimensional Constant Coefficient Case}

In this section we will consider a simplified version of the problem, where $X \subset \R^2$, and the coefficients $\sigma$ and $k$ are constant in $x$.  In this case, the angular variable $\theta$ now ranges over $S^1$, and we can provide a different proof of Theorem \ref{ControlTheorem} by analyzing the RTE in terms of Fourier series in the angular variable.  This simplified approach may have applications to numerical experiments.  It is possible that a similar approach with spherical harmonics works in three dimensions, but there the recursion relations are much more complicated, and it is not clear that the solutions converge.

If we parametrize $\theta \in S^1$ by writing $\theta = (\cos t, \sin t)$, for $t \in [0,2\pi]$, then the RTE from Theorem \ref{ControlTheorem} becomes
\begin{equation}\label{2DAdjointRTE}
\cos t \partial_x v + \sin t \partial_y v  + \sigma v = \int_0^{2\pi} k(t,t')v(x,t')\, d t'.
\end{equation}

We will prove the following version of Theorem \ref{ControlTheorem}.

\begin{prop}
For each $x_0 \in \R^2$ and $m \in \Z$, there exists a solution $v$ to \eqref{2DAdjointRTE} such that $v(x_0) = e^{imt}$ and $\|v\|_{L^{\infty}}$ is bounded uniformly in $m$.  
\end{prop}

\begin{proof}
We begin with the ansatz 
\begin{equation}
v(x,t) = \sum_{n \in \Z} v_n(x) e^{int}.
\end{equation}
By the assumptions on the regularity of $k$, we can expand $k(t,t')$ in terms of $\{e^{int}\}$ and $\{e^{int'}\}$ as well.  Crucially, the assumption that $k$ obeys reciprocity implies that the expansion has the form
\begin{equation}
k(t,t')= \sum_{n \in \Z} k_n e^{in(t-t')}.
\end{equation}
Therefore
\begin{equation}
\int_0^{2\pi} k(t,t')v(x,t')\, d t' = \sum_{n \in \Z} k_n v_n(x) e^{int}.
\end{equation}
Now we can rewrite \eqref{2DAdjointRTE} as
\begin{equation}
e^{it}\partial_z v + e^{-it}\partial_{\bar{z}} v + \sigma v = \int_0^{2\pi} k(t,t')v(x,t')\, d t',
\end{equation}
where $\partial_z = \half (\partial_x - i\partial_y)$ and $\partial_{\bar{z}} = \half(\partial_x + \partial_y)$. Substituting into the expansions for $v$ and $k$ and examining the $e^{int}$ term, gives the system of equations
\begin{equation}
\partial_z v_{n-1} + \partial_{\bar{z}}v_{n+1} = (k_n - \sigma) v_n.
\end{equation}
Now suppose $m \in \Z$. Let $v_{n} = 0$ for $n < m$, $v_m = 1$, and 
\begin{equation}\label{2Dvn}
v_n = \frac{\Pi_{j=m}^{n-1}(k_n - \sigma)}{(n-m)!}\bar{z}^{n-m}
\end{equation}
for $n > m$.  The equations for $e^{int}, n < m$ are automatically satisfied, since each term in those equations is zero.  For $n \geq m$, note that the $\partial_z v_{n-1}$ is always zero.  Therefore the equation for $e^{int}$ reads
\begin{equation}
 \partial_{\bar{z}}v_{n+1} = (k_n - \sigma) v_n,
\end{equation}
and one can check that it is satisfied by \eqref{2Dvn}.  

Lastly, the conditions on $k$ imply that $\{ k_n \}$ are bounded, so there exists some constant $C$ such that 
\begin{equation}
|v_n| \leq \frac{C^{n-m}}{(n-m)!}|z|^{n-m}.
\end{equation}
Therefore the sum
\begin{equation}
v(x,t) = \sum_{n \in \Z} v_n(x) e^{int}
\end{equation}
converges and is bounded uniformly in $m$ and $v(x - x_0, t)$ a smooth solution to \eqref{2DAdjointRTE}, such that $v(x_0, t) = e^{imt}$.
\end{proof}

\section*{Acknowledgements}

GB was supported in part by the NSF grant DMS-1108608.
JCS was supported in part by the NSF grants DMS-1115574 and DMS-1108969.


\begin{thebibliography}{99}

\bibitem{Acosta1} S. Acosta. Time reversal for radiative transport with applications to inverse and control problems. \emph{Inv. Prob.} \textbf{29}: 085014, 2013.

\bibitem{Acosta2} S. Acosta. Recovery of the absorption coefficient in radiative transport from a single measurement, \emph{Inv Prob \& Im.} \textbf{9}, 289-300, 2015.

\bibitem{ammari_2008}
H. Ammari, E. Bonnetier, Y. Capdeboscq, M. Tanter and M. Fink, SIAM J. Appl. Math. {\bf 68}, 1557-1573 (2008).

\bibitem{Bal_IP2009}
G. Bal. Inverse transport theory and applications, {Inverse Problems}, \textbf{25} (2009), p. 053001.

\bibitem{BalSchotland_2010}
G. Bal and J. C. Schotland, Phys. Rev. Lett. {\bf 104}, 043902 (2010).

\bibitem{bal_2012_a}
G.~Bal in Inside Out II, G. Uhlmann Editor (Cambridge University Press,
Cambridge, UK, 2012).

\bibitem{bal_2010}
G. Bal, G. Uhlmann, Inverse Problems {\bf 26}, 085010 (2010).

\bibitem{bal_2012_b}
G.~Bal and G.~Uhlmann, Comm. Pure Appl. Math. {\bf 66}, 1629-1652 (2013).

\bibitem{bal_2012_c}
G. Bal, E. Bonnetier, F. Monard and F. Triki, Inverse Problems and Imaging {\bf 7}, 353-375 (2013).

\bibitem{bal_2013}
G. Bal, W. Naetar, O. Scherzer and J.C. Schotland, J. Ill-Posed and Inverse Problems {\bf 21}, 265280 (2013).

\bibitem{BalSchotland_PhysRev2014}
G. Bal and J. Schotland.  Ultrasound modulated bioluminescence tomography. {Phys. Rev. E.} Vol 89, \textbf{3} (2014) no. 031201.

\bibitem{BalTamasan}
G. Bal and A. Tamasan.  Inverse source problems in transport equations.  {SIAM J. Math. Anal.} \textbf{39}, 1 (2007) p. 57-76.

\bibitem{capdeboscq_2009}
Y. Capdeboscq, J. Fehrenbach, F. de Gournay and O. Kavian, SIAM J. Imaging Sciences, {\bf 2}, 1003-1030 (2009).

\bibitem{Cessenat1985}
M. Cessenat.  Th\'{e}or\`{e}mes de trace pour des espaces de fonctions de la neutronique. {C.R. Acad. Sci. S\'erie I}, 300 (1985) 89-92.

\bibitem{Cessenat1984}
M. Cessenat.  Th\'{e}or\`{e}mes de trace $L^p$ pour des espaces de fonctions de la neutronique. {C.R. Acad. Sci. S\'erie I}, 299 (1984) 831-834.

\bibitem{ChoulliStefanov}
M. Choulli and P. Stefanov. An inverse boundary value problem for the stationary transport equation, {Osaka J. Math.}, \textbf{36} (1999), p. 87-104.

\bibitem{contag}
C. Contag and M. H. Bachmann, Annu. Rev. Biomed. Eng. {\bf 4}, 235-260 (2002).

\bibitem{DautrayLions}
R. Dautray and J.-L. Lions. {Mathematical Analysis and Numerical Methods for Science and Technology.}  Vol. 6, Springer Verlag, Berlin, 1993.

\bibitem{gebauer_2009}
B. Gebauer and O. Scherzer, SIAM J. Applied Math. {\bf 69}, 565-576 (2009).  

\bibitem{huynh_2013}
N. T. Huynh, B. R. Hayes-Gill, F. Zhang and S. P. Morgan,
J. Biomedical Optics {\bf 18}, 020505 (2013).

\bibitem{kuchment}
P. Kuchment and L. Kunyansky, J. Appl. Math. {\bf 19}, 191-224 (2008); ibid
Inverse Problems {\bf 27} 055013 (2011).

\bibitem{kuchment_2012}
P. Kuchment and D. Steinhauer, Inverse Problems {\bf 28}, 084007 (2012).

\bibitem{larsen}
E. W. Larsen. The inverse source problem in radiative transfer. J. Quant. Spect. Radiat. Transfer, 15:1-5, 1975.

\bibitem{monard_2011}
F. Monard and G. Bal, Inverse Problems and Imaging {\bf 6}, 289-313 (2012).

\bibitem{mclaughlin_2004}
J. R. McLaughlin and J. Yoon, Inverse Problems {\bf 20}, 2545 (2004).

\bibitem{mclaughlin_2010}
J. R. McLaughlin, N. Zhang and A. Manduca, Inverse Problems {\bf 26}, 085007 (2010).

\bibitem{nachman}
Adrian Nachman, Alexandru Tamasan and Alexandre Timonov, Inverse Problems {\bf 23}, 2551-2563 (2007);
ibid, Inverse Problems {\bf 25}, 035014 (2009).

\bibitem{vasilis}
V. Ntziachristos, J. Ripoll, L. H. V. Wang, and R. Weissleder, Nat. Biotech. {\bf 23}, 313-320 (2005).

\bibitem{siewert}
C. E. Siewert. An inverse source problem in radiative transfer. J. Quant. Spect. Radiat. Transfer, 50:603-609,
1993.

\bibitem{StefanovUhlmann}
P. Stefanov and G. Uhlmann. An inverse problem in optical molecular imaging.  {Analysis and PDE} \textbf{1}, (2008) p. 115-126.

\end{thebibliography}
\end{document}